\documentclass[11pt]{amsart}
\usepackage[T1]{fontenc}
\usepackage[cp1250]{inputenc}
\usepackage{amsmath}
\usepackage{amssymb}
\usepackage{amscd}
\usepackage[colorlinks]{hyperref}
\usepackage{tikz}
\usetikzlibrary{cd}
\usepackage{clock}
\usepackage{enumerate}
\usepackage[all]{xy}

\usepackage{graphicx}

\makeatletter
\DeclareRobustCommand{\cev}[1]{%
  {\mathpalette\do@cev{#1}}%
}
\newcommand{\do@cev}[2]{%
  \vbox{\offinterlineskip
    \sbox\z@{$\m@th#1 x$}%
    \ialign{##\cr
      \hidewidth\reflectbox{$\m@th#1\vec{}\mkern4mu$}\hidewidth\cr
      \noalign{\kern-\ht\z@}
      $\m@th#1#2$\cr
    }%
  }%
}
\makeatother

\newenvironment{pf}{\begin{proof}}{\end{proof}}

\newcommand{\ide}{\operatorname{id}}

\newcommand{\Pe}{{\mathbb{P}}}

\newcommand{\fK}{{\mathfrak{K}}}
\newcommand{\fL}{{\mathfrak{L}}}

\newcommand{\fC}{{\mathfrak{C}}}

\newcommand{\kfF}{{\mathfrak{F}_{\rm fin}}}
\newcommand{\kpF}{{\mathfrak{F}_{\rm pro}}}

\newcommand{\Ob}{\operatorname{Obj}}

\newcommand{\dst}{\star\star}

\newcommand{\subs}{\subseteq}

\renewcommand{\int}{\operatorname{int}}
\newcommand{\fra}{Fra\"iss\'e }

\newcommand{\amor}{\SelectTips{cm}{11}}

\newcommand{\map}[3]{#1\colon #2 \to #3} 
\newcommand{\pair}[2]{(#1, #2)} 
\newcommand{\img}[2]{#1[#2]} 
\newcommand{\obj}[1]{\operatorname{Obj}\left(#1\right)}
\newcommand{\nic}[1]{}
\newcommand{\id}[1]{{\operatorname{id}_{#1}}}

\newtheorem{theorem}{Theorem}[section]

\newtheorem{lemma}[theorem]{Lemma}

\newtheorem{proposition}[theorem]{Proposition}
\theoremstyle{definition}

\author{Wies{\l}aw Kubi\'s} 
\address{Wies{\l}aw Kubi\'s\\  Institute of Mathematics, Czech Academy of Sciences, Czechia\\
	Institute of Mathematics\\
	Cardinal Stefan Wyszynski University in Warsaw\\
	Warszawa, Poland} \email{kubis@math.cas.cz}
\author{Andrzej Kucharski} 
\address{Andrzej Kucharski \\ University of Silesia in Katowice \\ Katowice, Poland} \email{andrzej.kucharski@us.edu.pl} 
\author{ S\l awomir Turek}
\address{S\l awomir Turek\\ Institute of Mathematics\\
	Cardinal Stefan Wyszynski University in Warsaw\\
	Warszawa, Poland}
\email{s.turek@uksw.edu.pl}
\date{ver. 2026-08-18, \clocktime\today}

\title{Homogeneity of profinite graphs}

\begin{document}
	
	\maketitle
	
	\begin{abstract}
		Profinite graphs are inverse limits of sequences of nonempty finite graphs. There exists a ``generic'' graph $\Pe$ in this category, although its edge relation is rather poor: only isolated vertices and isolated edges. On the other hand, its topological structure (besides being homeomorphic to the Cantor set) is not trivial, as we show in this note.
		
		We prove a Knaster--Reichbach type result: Every topological isomorphism   $h\colon K\to L$ where  $K$ and  $L$ are nowhere dense closed subsets of $\Pe$ consisting of isolated vertices can be extended to a topological automorphism of the whole graph $\Pe$.
	\end{abstract}

	\section{Introduction}
	
A graph is typically called \emph{homogeneous} if every isomorphism between its small subgraphs extends to an automorphism, where ``small'' typically means ``finite'', although it could also refer to some special class of subgraphs (e.g. finite connected ones, singletons, etc.).
The well-known universal homogeneous graph is the \emph{Rado graph} (also called \emph{random}, due to a work of Erd\H os -- Renyi, independent of Rado). The Rado graph can be viewed as the most complicated countable graph (containing all countable graphs) and the most symmetric one, in the sense that every isomorphism between its finite subgraphs extends to an automorphism.
The Rado graph (as well as any countable graph) can be built as the union of a chain of finite graphs\footnote{In fact, Rado constructed his graph on the nonnegative integers, providing a quite natural and simple formula. Erd\H os and Renyi discovered it through a random process of building a chain of graphs.}, even in such a way that at each step exactly one new vertex is added. A dual procedure, where embeddings are replaced by \emph{quotient homomorphisms} (see the definition below) leads to \emph{profinite graphs}.
Recall that profinite graphs can be characterized either as limits of inverse sequences of finite graphs or as metrizable Boolean spaces with a closed edge relation. While in the classical graph theory the edge relation is typically irreflexive, in the inverse setting we require that it is reflexive, simply because we would like a singleton to be the image of a quotient homomorphism.

In the category of profinite graphs, a natural variant of homogeneity is: Every isomorphism between finite quotients lifts to a topological automorphism, namely, a graph automorphism that is at the same time a homeomorphism.
The profinite analogue of the Rado graph indeed exists, it is a graph $\Pe$, characterized uniquely (up to a topological isomorphism) by the following property:
\begin{enumerate}
	\item[(P)] Given a quotient homomorphism $\map f {G'}{G}$ between finite graphs, given a quotient homomorphism $\map q \Pe G$, there exists a quotient homomorphism $\map{q'}{\Pe}{G'}$ such that $q = f \circ q'$.
\end{enumerate}
Recall that in the case of infinite profinite graphs, quotient homomorphisms are required to be continuous.

As it happens, from the point of view of graph theory, the graph $\Pe$ is a bit disappointing: It consists of disjoint edges and a bunch of $0$-degree vertices. More precisely, $\Pe$ contains an uncountable (continuum-sized) collection of edges, none of them adjacent to another, and a dense $G_\delta$ set of vertices of degree zero (we call them \emph{isolated vertices}).
Clearly, $\Pe$ cannot be homogeneous with respect to points, namely, the canonical action of its automorphism group is not transitive, simply because there are vertices of degree one as well as vertices of degree zero.

We consider  undirected graphs with unique edges and with loops, in other words, structures with a symmetric reflexive binary relation.
Profinite graphs have been studied 
by several authors, see \cite{Camerlo},\cite{DroGoe},\cite{IrwSol}, \cite{KubFra},\cite{GGK}. The notion of a projective \fra family of topological structures has been introduced in \cite{IrwSol}. Irwin and Solecki gave a dualization of the \fra limit construction from model theory, see \cite{Fra}. Using projective \fra theory, Camerlo showed  (\cite{Camerlo}) that there exists a profinite graph $\Pe$  which is uniquely characterized by  a certain kind of projectivity with respect to finite graphs. It turns out that
$\Pe$ is homeomorphic to the Cantor set and it does not contain paths of length $>1$. Our approach fits within the framework of category-theoretic \fra limits, developed in \cite{DroGoe},\cite{KubFra},\cite{PechPech},\cite{PechPech2}. 
We prove that an embedding $\eta\colon K\to\Pe$ of
a fixed 0-dimensional compact space $K$ into $\Pe$  such that $\eta[K]$ is a nowhere dense and  $\eta[K]$ consisting of isolated vertices in $\eta[K]$ is
a \fra limit (generic object) in a suitable category. From this we derive   a Knaster--Reichbach type theorem for graph structures, i.e. every topological isomorphism   $h\colon K\to L$ between   nowhere dense closed sets $K,L\subset \Pe$ such that $K,L$ have only  isolated vertices  can be extended to a topological automorphism of the whole graph $\Pe$.

\section{Preliminaries}
	For undefined notions concerning category theory we refer to \cite{MacLane}. We will use the concepts and facts related to the theory of \fra sequences presented in~\cite{KubFra}, in the context of reversed categories.
We say that category $\fK$ is \emph{directed}
if for every $A,B\in\obj{\fK}$ there are $C\in\obj{\fK}$ and arrows $f\in\fK(C,A)$ and $g\in\fK(C,B)$.
We say that $\fK$ has the \emph{amalgamation property} if for every $A,B,C\in\obj{\fK}$ and for every morphisms $f\in\fK(B,A)$, $g\in\fK(C,A)$ there exist $D\in\obj{\fK}$ and morphisms $f'\in\fK(D,B)$ and $g'\in\fK(D,C)$ such that $f\circ f' = g\circ g'$.

A \textit{\fra sequence} in $\mathfrak{K}$ is an countable inverse sequence $\cev{u}$  satisfying the following conditions:
\begin{enumerate}
	\item[(U)] For every object $X$ of $\mathfrak{K}$ there exists $n<\omega$ such that $\mathfrak{K}(U_n, X)\ne\emptyset$.
	\item[(A)] For every $n<\omega$ and for every morphism $f\colon Y\to U_n$, where $Y\in \Ob(\mathfrak{K})$, there exist $k\ge n$ and $g\colon U_k\to Y$ such that $u_n^k=f\circ g$.
\end{enumerate} 
 
 We recall some basic results concerning \fra theory  developed by the first author~\cite{KubFra}.
 
\begin{theorem}[cf. \cite{KubFra} Corollary~3.8]\label{exists}
 Assume that $\mathfrak{K}$ is a directed category with the amalgamation property and countable (up to isomorphism). Then $\mathfrak{K}$ has a \fra sequence.
\end{theorem}

Let us assume that $\mathfrak{K}$ is a full subcategory of a bigger category $\mathfrak{L}$ such that
the following compatibility conditions are satisfied.
\begin{enumerate}
	\item [(L0)] All $\mathfrak{L}$-arrows are epimorphisms.
	\item [(L1)] Every inverse sequence  in $\mathfrak{K}$ has the limit in $\mathfrak{L}$.
	\item [(L2)] Every $\mathfrak{L}$-object is the limit of an inverse sequence in $\mathfrak{K}$.
	\item
	[(L3)] For every inverse sequence $\cev{x}=\{X_n\colon n<\omega\}$ in $\mathfrak{K}$ with $X = \lim \cev{x}$ in $\mathfrak{L}$, for every $\mathfrak{K}$-object $Y$,
	for every $\mathfrak{L}$-arrow $f\colon X \to Y$ there exist $n<\omega$ and a $\mathfrak{K}$-arrow $f'\colon  X_n \to Y$ such
	that $f = f'\circ  x_n,$ where $x_n\colon X\to X_n$ is the projection.
\end{enumerate}
	\amor
\labelmargin-{1.8pt}
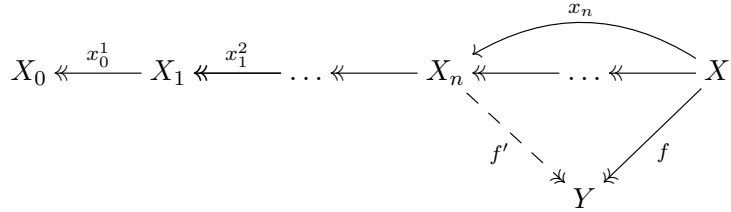
$$\xymatrix@R=2.8pc@C=2.7pc
{
	\labelmargin+{1pt}
	X_0& X_1\ar@{->>}[l]_{x^1_0}&\dots\ar@{->>}[l]_{x^2_1}\ar@{->>}[l]& X_n\ar@{-->>}[rd]_{f'}\ar@{->>}[l]&\dots\ar@{->>}[l]& X\ar@/_1.6pc/[ll]_{x_n}\ar@{->>}[dl]^f\ar@{->>}[l]\\
	   &                        & &&Y&
	}$$

Now, an $\mathfrak{L}$-object $U$ will be called {\em $\mathfrak{K}$-generic} if
\begin{enumerate}
	\item [(G1)] $\mathfrak{L}(U,X) \ne\emptyset$ for every $X \in \text{Obj} (\mathfrak{K})$.
	\item
	[(G2)] For every $\mathfrak{K}$-arrow $f\colon Y \to X$, for every $\mathfrak{L}$-arrow $g\colon U \to X$ there exists an
	$\mathfrak{L}$-arrow $h\colon U \to Y$ such that $f \circ h = g$.
\end{enumerate}

\begin{theorem}[\cite{KKT25}]\label{fra-gen}
	Assume that  $\mathfrak{K} \subs \mathfrak{L}$ satisfy {\rm(L0)}--{\rm(L3)} and $\mathfrak{K}$ has the amalgamation property. If  $\cev{u}$ is a \fra sequence in $\mathfrak{K}$, then $U=\lim\cev{u}$ is a $\mathfrak{K}$-generic.
\end{theorem}

A \textit{graph} is a structure of the form $G = (V(G), E(G))$, where $E(G)\subseteq V(G)^2$ is a symmetric and reflexive relation.
The elements of $E(G)$ are called \textit{edges} and the elements of $V(G)$ are called \textit{vertices}.
We say that vertices $x,y$ are \textit{adjacent} if $\pair xy \in E(G)$.
We shall often identify $G$ and $V(G)$. A vertex $v$ in a graph $G$ is an \textit{isolated vertex} if it is adjacent to itself only. 
A \textit{homomorphism of graphs} is a map $\map f G H$ that preserves edges. Given a graph $G$, an \textit{induced subgraph} is a subset $H \subs G$ with the induced graph relation, namely, $(x,y)\in E $ in $H$ if an only if $(x, y)\in E$ in $G$. A \textit{subgraph} often refers to a subset $A$ of $G$ such that the inclusion $A$ is a graph homomorphism. 

An \textit{embedding} $\map eGH$ is a one-to-one graph homomorphism such that whenever $\pair {e(x)}{e(y)}$ is adjacent in $H$, $\pair xy$ is adjacent in $G$ as well. That is, an embedding $\map eGH$ is an isomorphism between $G$ and the induced subgraph $\img eG$ of $H$.  

A graph homomorphism $\map fXY$ is \textit{strict} if for every adjacent $p,q$ in the range of $f$ there exist adjacent $a,b\in X$ such that $f(a)=p$ and $f(b)=q$. A strict homomorphism which is additionally surjective will be called a \textit{quotient map}.

A graph $G=(V(G),E(G))$ is \textit{profinite graph} if it is an inverse limit  of finite graphs $G_n=\pair {V(G_n)}{E(G_n)}$ with bonding maps $f^{n+1}_n:G_{n+1}\to G_n,$ which are quotient maps.  More precisely, $V(G)=\lim(V(G_n),f_n^{n+1})$, where each $V(G_n)$ is discrete space and $(x,y)\in E(G)$ iff $(x_n,y_n)\in E(G_n)$ for each $n\in\omega$.  
It is worth noting that the relation $E(G)$ is a closed subset of the product $V(G)\times V(G)$.  A graph isomorphism $f$ between profinite graphs $G$ and $H$ is a {\em topological isomorphism}, whenever $f$ is a homeomorphism.

\section{Categories of finite graphs}
	
We shall work with  the category $\kfF$  whose  objects are finite graphs and arrows are quotient maps, and  also we shall consider  the category $\kpF$  whose  objects are profinite graphs and arrows are quotient maps.

Camerlo \cite{Camerlo} using concept of a 
projective \fra families (see \cite{IrwSol}) studied graph and its relatives. It  was considered the category of $\kfF$  whose  obj ects are finite graphs and arrows are quotient maps. We shall recall some of results from \cite{Camerlo} 
into the framework of Droste \& G\"obel theory~\cite{DroGoe} or Kubi\'s~\cite{KubFra}.

\begin{theorem}[cf. Camerlo \cite{Camerlo}]\label{camerlo1}
There exists a \fra sequence  $\cev{g}=(G_n\colon n<\omega)$ in the category $\kfF$.  In particular,  the inverse limit $\Pe=\lim \cev{g}$ is the Cantor set as topological space and  
is a unique  profinite graph with respect to the following property:
\begin{enumerate}
	\item[$(\star)$]

	for every continuous quotient maps  $h\colon \Pe\to Z$  and $ f\colon Y\to Z$,  where  $Z,Y\in\kfF,$ there exists a $\kpF$-arrow $g\colon \Pe\to Y$  such that the following diagram commutes.
	
\end{enumerate} 	

\amor
$$\xymatrix{
	\Pe\ar@{->>}[d]_h \ar@{-->>}[dr]^g &\\
	Z & Y.\ar@{->>}[l]^f}$$
	
\end{theorem}

To prove our main theorem, we need some necessary lemmas on $\Pe$.

\begin{lemma}\label{l:3.2}
If  $A,B\subseteq\Pe$ are closed disjoint subsets such that $(a,b)\not\in E(\Pe)$ for all  $a\in A$ and $b\in B$, then there are disjoint clopen subsets $W_A,W_B$ such that $W_A\cup W_B=\Pe$ and $A\subseteq W_A$ and $B\subseteq W_B$ and $(a,b)\not\in E(\Pe)$ for all $a\in W_A$ and $b\in W_B.$
\end{lemma}
\begin{pf}
Fix $a\in A$ and $b\in B$. Since $E\subseteq\Pe^2$ is closed subset there are  disjoint clopen neighborhoods $U_a$ and $U_b$ of 
$a$ and $b$, respectively, such that $(x,y)\not\in E(\Pe)$ for all $x\in U_a$ and $y\in U_b$.  By 
 the compactness of the set $B$
  for any $a\in A$ there are  disjoint clopen neighborhoods $U_a$ and $U_B$ of 
$a$ and $B$, respectively, such that $(x,y)\not\in E(\Pe)$ for all $x\in U_a$ and $y\in U_B.$ Using this time the compactness of $A$ we conclude that there are disjoint clopen subsets $V_0,V_1$ such that $A\subseteq V_0$ and $B\subseteq V_1$ and $(x,y)\not\in E(\Pe)$ for all $x\in V_0$ and $y\in V_1.$ 

Let $S=\{0,1,2\}$. Define a graph relation $E(S)$ as follows $E(S)$ is symmetric and reflexive and  $(2,i)\in E(S)$ if and only if there are $x\in V_i$ and $y\in \Pe\setminus (V_0\cup V_1)$ such that $(x,y)\in E(\Pe)$, where $i\in\{0,1\}$ and $(0,1)\not\in E(S).$  Define $h\colon\Pe\to S$ by 
$$h(x)=
\begin{cases}
	0,& \mbox{ if } x\in V_0,\\
	1,& \mbox{ if } x\in V_1,\\
	2,&\mbox{ otherwise.}
\end{cases}$$
Clearly $h$ is quotient map. Let $T=\{0,1,3,4\}$. Define a graph relation $E(T)$ as follows $E(T)$ is symmetric and reflexive and $(i,3),(i,4)\in E(T)\Leftrightarrow(i,2)\in E(S)$ for $i\in\{0,1\}$ and $(0,1),(3,4)\not\in E(T).$
Let $f\colon T\to S$ be defined by  

$$f(x)=
\begin{cases}
	2,& \mbox{ if } x\in \{3,4\},\\
	x,&\mbox{ if }x\in \{0,1\}.
\end{cases}$$
It's easy to see that  $f$ is quotient. By property $(\star)$ there is a quotient map $g\colon \Pe\to T$ such that 
$h=g\circ f.$ Since $(3,4)\not\in E(T)$ the sets $g^{-1}(3),g^{-1}(4)$ are clopen disjoint sets such that $(a,b)\not\in E(\Pe)$ for all $a\in g^{-1}(3)$ and $b\in g^{-1}(4).$ Let $W_A=V_0\cup g^{-1}(3)$ and $W_B=V_1\cup g^{-1}(4)$. Therefore $A\subseteq W_A$ and $B\subseteq W_B$  are closed disjoint subsets such that $(a,b)\not\in E(\Pe)$ for all $a\in W_A$ and $b\in W_B$ and $W_A\cup W_B=\Pe.$
\end{pf}

\begin{lemma}\label{l:3.3}
If 	$W\subseteq \Pe$ is clopen non-empty subset, then 
there are two points $a,b\in W$ with $(a,b)\in E(\Pe).$ 
\end{lemma}
\begin{pf}
Let $W$ is clopen non-empty subset of $\Pe$. Let $S=\{0,1\}$. We can assume that $\Pe\setminus W$ is non-empty set, since $\Pe$ contains non-trivial edge. We define a graph relation $E(S)$ as follows $E(S)$ is symmetric and reflexive and $(0,1)\in E(S)$ if and only if there are $x\in W$ and $y\in \Pe\setminus W$ such that $(x,y)\in E(\Pe)$. Define $h\colon \Pe\to S$ in the following way 

$$h(x)=
\begin{cases}
	0,& \mbox{ if } x\in W,\\
	1,&\mbox{ otherwise.}
\end{cases}$$
Clearly, $h$ is  quotient map. Let $T=\{0,1,2,3\}$ and define a graph relation $E(T)=E(S)\cup \{2,3\}^2.$ Let $f\colon T\to S$ be defined by  

$$f(x)=
\begin{cases}
	0,& \mbox{ if } x\in \{0,2,3\},\\
	1,&\mbox{ otherwise.}
\end{cases}$$
We see at once that $f$ is quotient map. By  $(\star)$ we get a quotient map $g\colon \Pe\to T$ such that 
$h=f\circ g.$ Since $(2,3)\in E(T)$  there are two vertices  $a,b\in \Pe$ such that $(a,b)\in E(\Pe)$ and $g(a)=2$ and $g(b)=3$. It remains to notice that $a,b\in W$. 
\end{pf}

Comma categories in \fra theory have already been successfully used in ~\cite{PechPech}, \cite{PechPech2},\cite{BKW}, \cite{KKT25}.	
 Fix a profinite space $K$, i.e.   a nonempty compact 0-dimensional second countable topological space.

We define the category $\fL_K$ as follows. The objects of $\mathfrak{L}_K$ are continuous mappings $f\colon  K\to X$, with $X\in\kpF$.
Given two $\mathfrak{L}_K$-objects $f_0\colon  K \to X_0$, $f_1 \colon  K \to X_1$, an $\mathfrak{L}_K$-arrow from $f_1$
 to $f_0$ is a quotient continuous map $q \colon  X_1\to X_0$ satisfying $q \circ  f_1 = f_0$. The composition in 
 $\mathfrak{L}_K$ is the usual composition of mappings.  Let $\mathfrak{C}_K$  be the full
subcategory of $\mathfrak{L}_K$ whose objects are continuous mappings $f \colon  K \to X,$ with $X\in\kfF.$

\begin{lemma}\label{l:3.4}
	The category $\mathfrak{C}_K$ is countable (up to isomorphism) and directed and has the amalgamation property. 
	Moreover $\fC_K\subseteq\fL_K$ satisfy $(L0)-(L3)$.
\end{lemma}
\begin{pf}
The proof of this lemma is similar to the proof of \cite[Lemma 3.2]{KKT25}, we just need to take care of the 
graph structure.
Let $f\colon K\to X$,  $g\colon K\to Y$, $h\colon K\to Z$ be  $\fC_K$-objects and $q_1\colon f\to h$, $q_2\colon g\to h$ 
be $\fC_K$-arrows. We  get the following equations $q_1\circ f=h=q_2\circ g$. Consider  
the space $W=\{(x,y)\in X\times Y\colon q_1(x)=q_2(y)\}$   with surjections $f_1\colon W\to X$ and $g_1\colon W\to Y$ such that $q_1\circ f_1=q_2\circ g_1$ and define the graph structure on 
$W$ in the following way: 
		$$((a_1,b_1),(a_2,b_2))\in E(W)\Leftrightarrow (a_1,a_2)\in E(X)\;\&\; (b_1,b_2)\in E(Y),$$
		where $(a_1,b_1),(a_2,b_2)\in W.$ 
	\SelectTips{cm}{11}
	$$\xymatrix{
		& &X\ar@{->>}[dr]^{q_1} &\\
		W\ar@{->>}[rru]^{f_1}\ar@{->>}[drr]_{g_1} & K\ar[ur]_f\ar[rd]^g\ar[rr] ^h\ar[l]_(.36){k}&& Z\\
		&& Y\ar@{->>}[ur]_{q_2} & }$$ 
	
From the pullback property it follows that exists a unique $k\colon K\to W$ such that  $f=f_1\circ k$ and $g=g_1\circ k$.
 If $((a_1,b_1),(a_2,b_2))\in E(W)$ then  $(f_1((a_1,b_1)),f_1((a_2,b_2)))=(a_1,a_2)\in E(X)$ and 
 $(g_1((a_1,b_1)),g_1((a_2,b_2)))=(b_1,b_2)\in E(Y)$. On the other hand if $(b_1,b_2)\in E(Y)$ there are $a_1',a_2'\in X$ 
 such that $q_1(a_1')=q_2(b_1)$ and $q_1(a_1')=q_2(b_1)$. Since $q_1$ is quotient map there are $a_1,a_2\in X$ such that 
 $(a_1,a_2)\in E(X)$ and $q_1(a_1)=q_1(a_1')$ and $q_1(a_2)=q_1(a_2')$. Thus we get $((a_1,b_1),(a_2,b_2))\in E(W)$. This proves  that the category $\fC_K$ has the  amalgamation property.
 
	

	In order to show that $\fC_K$ is directed, let $f\colon K\to X$ and 
	$g\colon K\to Y$ be objects of $\fC_K.$
	Consider $Z=X\times Y$ with discrete topology. Define the graph structure on 
	$Z$ in the following way:

$$((a_1,b_1),(a_2,b_2))\in E(Z)\Leftrightarrow (a_1,a_2)\in E(X)\;\& \;(b_1,b_2)\in E(Y)$$ 
	
	and let $h=(f,g)$.  Obviously projections $\pi_X\colon X\times Y\to X$ and $\pi_Y\colon X\times Y\to Y$ are quotient maps such that $\pi_X\circ h=f$ and $\pi_Y\circ h=g$, this complete the proof that $\fC_K$ is directed.
	
We shall have established the lemma if we	prove that $\fC_K\subset\fL_K$ fulfills condition $(L0)-(L3)$. It's easy to check that  $\fC_K\subseteq\fL_K$ satisfy $(L0)-(L2).$  What is left is to show  $(L3)$. Let $f\colon X\to Y$ be a $\fL_K$-arrow and $Y$ be 
a	$\fC_K$-object.  We find $n<\omega $ such that  $x_n[f^{-1}(y)]\cap x_n[f^{-1}(z)]=\emptyset$  for all $y\ne z$ and $y,z\in Y.$  Define a map $f'\colon X_n\to Y$ by $f'(x)=y\Leftrightarrow x\in x_n[f^{-1}(y)]$ for $y\in Y$ and $x\in X_n.$ Thus $f'$ is quotient and $f'\circ x_n=f.$  
\end{pf}

\begin{proposition}\label{prop:1}
	There exists a \fra sequence in $\mathfrak{C}_K.$
\end{proposition}
\begin{pf}
By Theorem \ref{exists} and Lemma \ref{l:3.4} 	there exists a \fra sequence in $\mathfrak{C}_K.$
\end{pf}

\begin{theorem}\label{fra-sequ}
	Assume that  $\cev{\phi}=(\phi_n\colon n<\omega)$  is a  \fra sequence in $\mathfrak{C}_K$, where $\phi_n\colon K\to U_n$ for each $n<\omega.$  Then 
	\begin{enumerate}
		\item[{\rm(1)}] $\cev{u}=(U_n\colon n<\omega)$ is a {\fra}sequence in the category $\kfF$.
		\item[{\rm(2)}]The limit map $\phi_{\omega}\colon K\to \lim \cev{u}$ has a left inverse i.e. there is $r\colon \lim \cev{u}\to K$
		such that $r\circ \phi_{\omega}=\ide_K$,
		\item[{\rm(3)}] The image $\phi_{\omega}[K]$ is a nowhere dense set and the following condition is fulfilled
		\begin{enumerate}
			\item[$(\dst)$] for all disjoint closed subsets $A,B\subseteq \phi_{\omega}[K]$   there are disjoint clopen sets $W_A,W_B\subseteq \Pe$ such that $W_A\cup W_B=\Pe$ and $A\subseteq W_A,B\subseteq W_B$ and   $(a,b)\not\in E(\Pe)$ for all $a\in W_A$ and $b\in W_B.$
		\end{enumerate}
		
	\end{enumerate}
\end{theorem}

\begin{pf}(1)
	It is clear that the sequence $\cev{\phi}$ induces the sequence $\cev{u}=(U_n\colon n<\omega)$ in the category $\kfF$.  According to the definition, it must be shown that the sequence $\cev{u}$ satisfies the conditions (U) and (A).
	
	First we show property (U).  Let $X\in \kfF$. Choose a point $a\in X$ and define  a map $f:K\to X$ as follows $f(x)=a$ for all $x\in X.$ The map $f$ is the  object in $\mathfrak{C}_K$. Since $\cev{\phi}$  is a  \fra sequence in $\mathfrak{C}_K$ there is $n<\omega$ and quotient map $g:U_n\to X$ such that $g\circ \phi_n=f.$
	
	Now we prove condition $(A)$. Let $Y\in\kfF$ be a finite graph and $p\colon Y\to U_n$ be a quotient map. If we forget a structure of graph of $Y, p$ and $U_n$ then  define a continuous map $g\colon K\to Y$ as follows $g(x)=y_x$ for all $x\in Y,$ where $y_x$ is chosen form  $\phi_n^{-1}(p(x))$. Then the map $g$ satisfies  $p\circ g=\phi_n$ and  it turns out that $g$ is  a  $\mathfrak{C}_K$-arrow. Thus there exists $k\ge n$ and a morphism $h\colon \phi_k\to g$ such that $p\circ h=u_n^k$ because $\cev{\phi}$ is a {\fra}sequence in $\mathfrak{C}_K$.

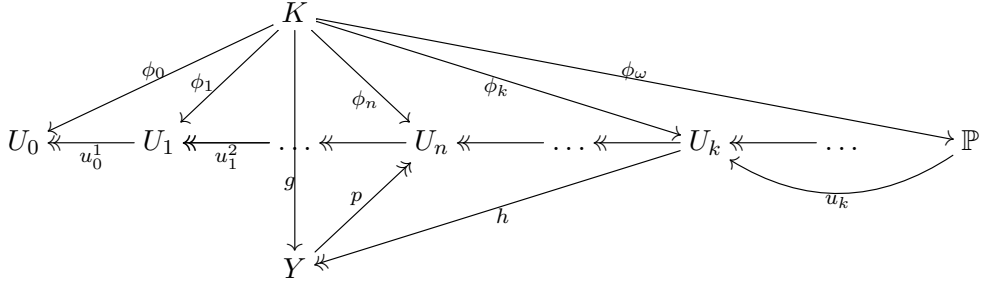
	\amor
	\labelmargin-{1.8pt}
	$$\xymatrix@R=2.8pc@C=2.7pc
	{&&K\ar[dll]_{\phi_0}\ar[dl]_(.6){\phi_1}\ar[dr]_(0.6){\phi_n}\ar[drrr]_{\phi_k}\ar[dd]_(.66)g\ar[drrrrr]^
		{\phi_{\omega}}&&&\\ 
		\labelmargin+{1pt}
		U_0& U_1\ar@{->>}[l]^{u^1_0}&\dots\ar@{->>}[l]^{u^2_1}\ar@{->>}[l]& U_n\ar@{->>}[l]&\dots\ar@{->>}[l]& U_k\ar@{->>}[dlll]^h\ar@{->>}[l]&\dots\ar@{->>}[l]&\Pe\ar@/^1.6pc/[ll]^{u_k}\\
		&& Y\ar@{->>}[ru]^p&&&
	}$$
	
	(2) Let $K'=K$ is equal as topological space and $K'$ consists of isolated vertices. The identity map $\ide\colon K\to K'$ is an object of $\mathfrak{L}_K$, and $\mathfrak{C}_K\subseteq \mathfrak{L}_K$ satisfies condition (L0)-(L3). By Theorem \ref{fra-gen} and condition $(G_1)$ there is quotient continuous map  $r\colon \Pe\to K'$	such that $r\circ \phi_{\omega}=\ide_K,$ (see also \cite{BKW})
	
	(3)  Suppose that  that the image $\phi_{\omega}[K]$ is not a nowhere dense set.  There are $n\in\omega$ and  
	$x\in U_n$ such that $u_n^{-1}(x)\subseteq \phi_{\omega}[K].$ Let $Y=U_n\times\{0\}\cup(\{x\}\times \{1\})$ and
	$((x,0)(y,0))\in E(Y)\Leftrightarrow (x,y)\in E(U_n)$ and $(x,1)$ is the isolated vertex. Define a map 
	$g\colon K\to Y$ as follows $g(x)=(\phi_n(x),0)$ and let $p$ be projection, $p(y,i)=y$ for all $(y,i)\in Y.$ Since 
	$g\colon K\to Y$ is an object of $\mathfrak{L}_K$ and $p$ is quotient map and $\phi_\omega\colon K\to\Pe$  is  $\mathfrak{C}_K$-generic, there is quotient map $f\colon \Pe\to Y$ such that $f\circ \phi_\omega=g.$ Note that
	$$f^{-1}(p^{-1}(\{x\}))=u_n^{-1}(x)\subseteq \phi_\omega[K].$$
	Thus $\phi_\omega^{-1}(f^{-1}(\{(x,1)\}))\ne\emptyset,$ but $\emptyset=g^{-1}(\{(x,1)\})= \phi_\omega^{-1}(f^{-1}(\{(x,1)\}))$ a contradiction.
	 
	In order to prove $(\dst)$ let $A,B\subseteq \phi_{\omega}[K]$ be disjoint closed subsets. There are disjoint clopen sets $U_A, U_B\subseteq K$ such that $U_A\cup U_B=K$ and $A\subseteq\phi_\omega[U_A]$ and $B\subseteq\phi_\omega[U_B]$.  Let $F=\{0,1\}$ be a discrete space and $F$ as a graph  consists of isolated vertices.
	Let $f\colon K\to F$ be defined as follows 
	$$f(x)=
	\begin{cases}
		0 &  \mbox{ if }x\in U_A,\\
		1 &  \mbox{ otherwise.}	
	\end{cases}$$ 
	The  map $f\colon K\to F$ is an object of $\mathfrak{L}_K$, and $\mathfrak{C}_K\subseteq \mathfrak{L}_K$ satisfy (L0)-(L3). By Theorem \ref{fra-gen} and condition $(G_1)$ there is quotient continuous map  $g\colon \Pe\to F$	such that $g\circ \phi_{\omega}=f.$ Since $g$ is quotient continuous map the sets $W_A=g^{-1}(0)$ and $W_B=g^{-1}(1)$ have the desired properties. 
\end{pf}

\begin{theorem}\label{generic}
Let $\eta\colon K\to\Pe$ be an embedding such that $\eta[K]$ is a nowhere dense and  $\eta[K]$ consisting of isolated vertices in $\eta[K].$  Then $\eta\colon K\to\Pe$ is $\mathfrak{C}_K$-generic.
\end{theorem}
\begin{proof}
Since $\mathfrak{C}_K$ has the terminal object 	it's  enough to prove  condition (G2), i.e. given  $X,Y\in\kfF$, a quotient map $f\colon X\to Y$,
	 a continuous map $b\colon K\to Y,$ and a continuous surjection $g\colon \Pe\to X$ such that $g\circ\eta=f\circ b$, there exists a quotient map $h\colon \Pe\to Y$ such that
	$f\circ h=g$   and $h\circ \eta=b$, that is the following diagram 
		\amor
	$$\xymatrix@R=2.5pc@C=2.9pc{
		K\ar[d]_{b}\ar@{^(->}[r]^{\eta}&\Pe\ar@{..>>}[dl]_(.49){h}\ar@{->>}[d]^{g}\\
		Y\ar@{->>}[r]_{f}&X}$$commutes.
		We can assume that there is only one $x_0\in X$ such that $|f^{-1}(x_0)|=2$ and $f$ is one-to-one on $Y\setminus f^{-1}(x_0)$, because every quotient map between finite graphs is the composition of finitely many such maps.  
		
		Let $f^{-1}(x_0)=\{y_0,y_1\}$ and  $W=g^{-1}(x_0)$ and $F_0=\eta[b^{-1}(y_0)], F_1=\eta[b^{-1}(y_1)]$. The sets $F_0,F_1$ are closed disjoint of   $W\cap \eta[K]$, by Lemma \ref{l:3.2} there are disjoint clopen sets $W_0',W_1'\subseteq \Pe$ such that $W_0'\cup W_1'=\Pe$ and $F_i\subseteq W_i'$ and   $(a,b)\not\in E(\Pe)$ for all $a\in W_0'$ and $b\in W_1'.$ Let $W_i=W\cap W_i'$.
		
		There are two cases: either $(y_0,y_1)\not\in E(Y)$ or
		$(y_0,y_1)\in E(Y).$
		
		(1)  If $(y_0,y_1)\not\in E(Y),$ then we define $h$ as follows
		
		$$h(x)=\begin{cases}
			f^{-1}(g(x)),&\mbox{ if } x\in \Pe\setminus W,\\
			y_0,&\mbox{ if } x\in W_0,\\
			y_1,&\mbox{ if } x\in W_1.
		\end{cases}$$
		(2)  Assume that $(y_0,y_1)\in E(Y).$ Since $\eta[K]$ is the nowhere dense set there is clopen non-empty set 
		$W_2\subseteq W$ and $W_2\cap \eta[K]=\emptyset.$ By Lemma \ref{l:3.2} there are $ a,b\in W_2$ such that 
		$(a,b)\in E(\Pe).$  Let $U_a,U_b$ be clopen disjoint neighborhood of $a,b$, respectively, such that $U_a\cup U_b=W_2.$ Then we define $h$ as following
		
		$$h(x)=\begin{cases}
			f^{-1}(g(x)),&\mbox{ if } x\in \Pe\setminus W,\\
			y_0,&\mbox{ if } x\in (W_0\setminus W_2)\cup U_a,\\
			y_1,&\mbox{ if } x\in (W_1\setminus W_2)\cup U_b.
		\end{cases}$$

		It is clear, in both cases, that $h$ is continuous quotient and $f\circ h=g$. It remains to check that $h\circ \eta=b.$ Let $x\in K$. We can assume $b(x)=y_0$ or $b(x)=y_1$. Otherwise, it is obvious that $h(\eta(x))=b(x).$ So,  if $b(x)=y_i$ then $\eta(x)\in W_i'\subseteq W_i$ and $h(\eta(x))=y_i=b(x)$ for each $i=0,1$.
		
\end{proof}
In the next theorem we show that $\Pe$ has the Knaster--Reichbach type property \cite{KnaRei}.

\begin{theorem}\label{main}  Every topological isomorphism $h\colon K\to L$ of nowhere dense closed subsets $K, L$ of  $\Pe$  such that  each $K, L$ are consisting of isolated vertices in $K$ and $L$, respectively,	can be extended to a  topological isomorphism of  $\Pe$.
\end{theorem}	
\begin{pf}
	Let $h\colon K\to L$ be a topological isomorphism of nowhere dense closed subsets of $\Pe$ such that each $K, L$ are consisting of isolated vertices in $K$ and $L$.
	Denote by $\eta\colon K\hookrightarrow \Pe$ 
	and $\eta_1\colon L\hookrightarrow \Pe$ the inclusion mappings. By Theorem \ref{generic}, both $\eta$ and $\eta_1\circ h$ are $\mathfrak{C}_K$-generic, therefore by uniqueness (\cite[Theorem 4.6]{KubFra}) there
	exists an $\mathfrak{L}_K$-isomorphism $H \colon  \Pe\to\Pe$ from $\eta$ to $\eta_1\circ h$. This means that there exists an  topological isomorphism of $\Pe$ extending $h$.
\end{pf}

\begin{theorem}
		Let $K$ be a  profinite graph.
	Then there exist a topological embedding $ \eta\colon K \to\Pe$ and a  continuous quotient map $ f\colon \Pe\to K$ such that $ \eta[K]$ is nowhere dense set and $\eta[K]$ consisting of isolated vertices in $\eta[K]$ and $f \circ \eta = \id K$.
\end{theorem}
\begin{pf}
	By Proposition \ref{prop:1} there exists a \fra sequence in $\mathfrak{C}_K.$ Using Theorem \ref{fra-sequ} we get 
	a topological embedding $ \eta\colon K \to\Pe$ and a  continuous quotient map $ f\colon \Pe\to K$ such that $ \eta[K]$
	 is nowhere dense set and $\eta[K]$ consisting of isolated vertices in $\eta[K]$. The identity map $\id K$ form 
	 profinite space $K$ to profinite graph $K$ is an object of $\mathfrak{L}_K$. Therefore by Theorem \ref{generic}  and 
	 (G1) there exists a $\mathfrak{L}_K$-arrow $f\colon \Pe\to K$ form $\eta$ to $\id K$.
	\end{pf}

\end{document}